\documentclass[12pt,reqno]{article}

\usepackage[usenames]{color}
\usepackage{amssymb}
\usepackage{amsmath}
\usepackage{amsthm}
\usepackage{amsfonts}
\usepackage{amscd}
\usepackage{graphicx}

\usepackage[colorlinks=true,
linkcolor=webgreen,
filecolor=webbrown,
citecolor=webgreen]{hyperref}

\definecolor{webgreen}{rgb}{0,.5,0}
\definecolor{webbrown}{rgb}{.6,0,0}

\usepackage{color}
\usepackage{fullpage}
\usepackage{float}

\usepackage{graphics}
\usepackage{latexsym}
\usepackage{epsf}
\usepackage{breakurl}

\def\Enn{\mathbb{N}}
\def\suchthat{\, : \,}

\newenvironment{smallarray}[1]
{\null\,\vcenter\bgroup\scriptsize
\arraycolsep=.13885em
\hbox\bgroup$\array{@{}#1@{}}}
{\endarray$\egroup\egroup\,\null}

\begin{document}

\theoremstyle{plain}
\newtheorem{theorem}{Theorem}
\newtheorem{corollary}[theorem]{Corollary}
\newtheorem{lemma}[theorem]{Lemma}
\newtheorem{proposition}[theorem]{Proposition}

\newtheorem{definition}[theorem]{Definition}
\theoremstyle{definition}
\newtheorem{example}[theorem]{Example}
\newtheorem{conjecture}[theorem]{Conjecture}

\theoremstyle{remark}
\newtheorem{remark}[theorem]{Remark}

\title{A Dombi Counterexample with Positive Lower Density}

\author{
Jeffrey Shallit\\
School of Computer Science\\
University of Waterloo \\
Waterloo, ON  N2L 3G1 \\
Canada\\
\href{mailto:shallit@uwaterloo.ca}{\tt shallit@uwaterloo.ca} }

\maketitle

\begin{abstract}
Let $r(k,A,n)$ denote the number of representations of $n$ 
as a sum of $k$ elements of a set $A \subseteq \Enn$.  
In 2002, Dombi conjectured that if $A$ is co-infinite, then
the sequence $(r(k,A,n))_{n \geq 0}$ cannot be strictly increasing.
Using tools from automata theory and logic,
we give an explicit
counterexample where $\Enn \setminus A$ has positive lower density.
\end{abstract}

\section{Introduction}
Let $\Enn = \{0,1,\ldots \}$ be the natural numbers, and let
$A \subseteq \Enn$.  Define
$r(k,A,n)$ to be the number of $k$-tuples of elements of $A$
that sum to $n$.  Dombi \cite{Dombi:2002} conjectured that
there is no infinite set $F$ such that $r(3,\Enn\setminus F,n)$ is
strictly increasing.   Recently Bell et al.~\cite{Bell&Shallit:2023}
found a counterexample to this conjecture.  However, the 
$F$ of their example is quite sparse; it has upper density $0$.
In this note we give a simple explicit example of an $F$ such that
$r(3,\Enn\setminus F,n)$ is strictly increasing and
$F$ has positive lower density.   The novelty to our approach
is the use of tools from automata theory and logic.

\section{The example}
Let $F = \{3,12,13,14,15,48,49,50, \ldots \}$ be the set of natural
numbers whose base-$2$ expansion is of even length and
begin with $11$.   This is an example of an {\it automatic set\/} \cite{Allouche&Shallit:2003};
that is, there is a finite automaton accepting exactly the 
base-$2$ expansions of the numbers of $F$.  It is depicted in
Figure~\ref{figf}.  Here $0$ is the initial state, and $3$ is the only accepting
state.  The input is a binary representation of $n$, starting with the most significant
digit.
\begin{figure}[H]
\begin{center}
\includegraphics[width=5.5in]{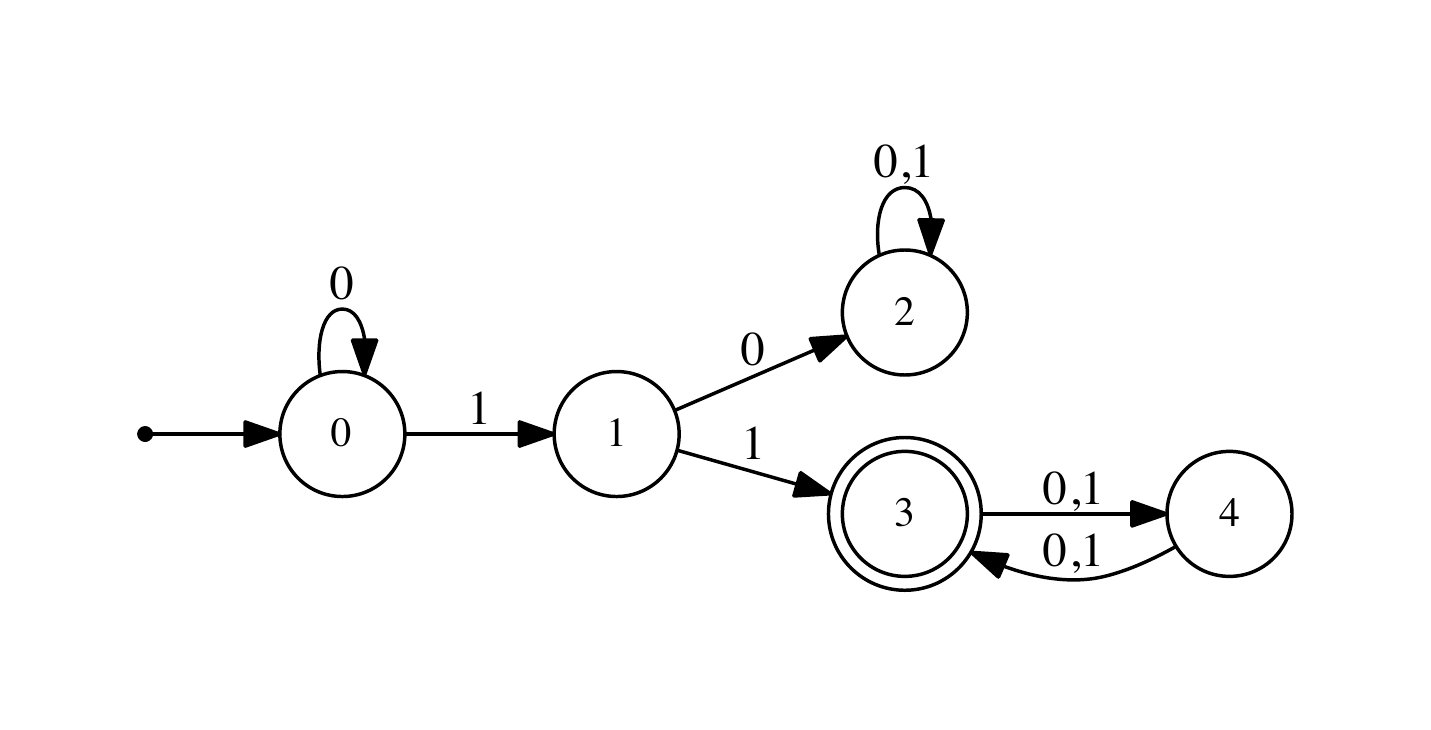}
\end{center}
\vskip -.5in
\caption{Automaton for $F$.}
\label{figf}
\end{figure}

For a set $X \subseteq \Enn$,
define $D_X (n) = {1\over n}| \{ X\, \cap \,
\{0,1,\ldots, n-1 \} \} | $.
Recall that the {\it lower density\/} of $X$
is defined to be $\liminf_{n \rightarrow \infty} D_X(n)$
and the {\it upper density\/} is
$\limsup_{n \rightarrow \infty} D_X(n)$.

\begin{proposition}
The lower density of $F$ is $1/9$ and the upper density is $1/3$.
\end{proposition}

\begin{proof}
The characteristic sequence of $F$ is
$$ 000\, 1 \, \overbrace{\{0 \cdots 0\}}^8 
\overbrace{\{1 \cdots 1\}}^4
\overbrace{\{0\cdots 0\}}^{32}
\overbrace{\{1\cdots1\}}^{16}
\cdots
\overbrace{\{0\cdots 0\}}^{2\cdot 4^n}
\overbrace{\{1\cdots 1\}}^{4^n}
\cdots . $$
So the lower density of $F$ is 
$$\liminf_{n \rightarrow \infty} D_X (3 \cdot 4^n) = 
{{1 + 4 + 16 + \cdots + 4^{n-1}} \over {3 \cdot 4^n}} =
{{{4^n -1} \over 3} \over {3 \cdot 4^n}} = {1 \over 9} ,$$
and the upper density is 
$$\liminf_{n \rightarrow\infty} D_X(4^n) 
= {{1 + 4 + 16 + \cdots + 4^{n-1}} \over {4^n}} = 
{{{4^n -1} \over 3} \over {4^n}} = {1 \over 3}.$$
\end{proof}

\begin{theorem}
The sequence $r(3,\Enn\setminus F,n)$ is strictly increasing.
\end{theorem}

\begin{proof}
Here is an outline of the proof.  Define $A := \Enn\setminus F$  and
$d(n) = r(3,A,n) - r(3,A,n-1)$.
We will show that $d(n) > 0$ for all $n$. To do this,
we show 
\begin{equation}
d(n) \geq 4d(\lfloor n/4 \rfloor) - 18.
\label{ineq}
\end{equation}
and then use an easy induction.

To prove the bound \eqref{ineq},
we show that $f(n) := d(n) - 4d(\lfloor n/4 \rfloor)$ is an {\it automatic
sequence} \cite{Allouche&Shallit:2003},
meaning that there is a deterministic finite automaton
with output (DFAO) computing $f$, and we explicitly determine
the automaton.   Once we have the automaton for $f$, we can
determine the range of $f$ simply by examining the (finitely many)
outputs associated with the states.

To find the DFAO for $f$, we first observe that 
$A := \Enn \setminus F$ is an automatic set since $F$ is.  Then the claim that
$n$ is the sum of three elements of $A$ is first-order expressible,
and hence by a theorem of B\"uchi and Bruy\`ere \cite{Bruyere&Hansel&Michaux&Villemaire:1994}, the set 
$$G := \{ (n,i,j,k) \suchthat n = i+j+k \text{ for } i,j,k \in A \}$$
is also automatic.  The automaton for $G$ can be computed explicitly
by free software called {\tt Walnut} \cite{Mousavi:2016,Shallit:2022}, with the following commands:
\begin{verbatim}
morphism x "0->01 1->23 2->22 3->44 4->33":
morphism y "0->0 1->0 2->0 3->1 4->0":
promote X x:
image FF y X:
def g "FF[i]=@0 & FF[j]=@0 & FF[k]=@0 & n=i+j+k":
\end{verbatim}
The resulting automaton has $143$ states.

Next, we use the fact that the number $r(3,A,n)$
of triples $(i,j,k)$ corresponding
to a particular $n$ has a {\it linear representation\/} 
\cite{Berstel&Reutenauer:2011}.  This
means there exist a row vector $v$, a matrix-valued morphism $\gamma$, and
a column vector $w$, such that $r(3,A,n) = v \gamma(x) w$
for all binary strings $x$ evaluating to $n$ (when $x$ is considered as a 
number in base $2$).
The dimension of $v$ is called the {\it rank\/} of the linear
representation.
Furthermore, this linear representation is computable
from the automaton for $G$ by the {\tt Walnut} command
\begin{verbatim}
def r3an n "$g(i,j,k,n)":
\end{verbatim}

In the same way we can compute linear representations for
$r(3,A,n-1)$, $r(3, A, \lfloor n/4 \rfloor)$, and
$r(3,A, \lfloor n/4 \rfloor - 1)$ using the commands
\begin{verbatim}
def r3anm1 n "$g(i,j,k,n-1)":
def r3an4 n "$g(i,j,k,n/4)":
def r3an4m1 n "$g(i,j,k,n/4-1)":
\end{verbatim}
These have rank $143, 446, 446$ respectively.
From these four linear representations, using a simple construction involving block matrices,
we can compute a linear representation for
$$f(n) := d(n) - 4d(\lfloor n/4 \rfloor) = r(3,A,n) - r(3,A,n-1) -
4(r(3,A,\lfloor n/4\rfloor) - r(3,A, \lfloor n/4 \rfloor - 1)).$$
The resulting linear representation has rank $1178$.

Next, we can use an algorithm due to Sch\"utzenberger \cite[Chap.~2]{Berstel&Reutenauer:2011} to minimize this linear representation, resulting in a
linear representation $(v',\gamma', w')$ for $f$ of rank $16$.  We give it explicitly below:

\begin{align*}
&{v'}^T = \left[ \begin{smallarray}{c}
1\\
 0\\
 0\\
 0\\
 0\\
 0\\
 0\\
 0\\
 0\\
 0\\
 0\\
 0\\
 0\\
 0\\
 0\\
 0
\end{smallarray} \right] 
\quad
\gamma'(0)  = {1\over{276}} \left[ \begin{smallarray}{ccccccccccccccccccc}
276&    0&    0&    0&    0&    0&    0&    0&    0&    0&    0&    0&    0&    0&    0&    0\\
0&    0&  276&    0&    0&    0&    0&    0&    0&    0&    0&    0&    0&    0&    0&    0\\
0&    0&    0&    0&  276&    0&    0&    0&    0&    0&    0&    0&    0&    0&    0&    0\\
0&    0&    0&    0&    0&    0&  276&    0&    0&    0&    0&    0&    0&    0&    0&    0\\
0&    0&    0&    0&    0&    0&    0&    0&  276&    0&    0&    0&    0&    0&    0&    0\\
0&    0&    0&    0&    0&    0&    0&    0&    0&    0&  276&    0&    0&    0&    0&    0\\
0&    0&    0&    0&    0&    0&    0&    0&    0&    0&    0&    0&  276&    0&    0&    0\\
0&    0&    0&    0&    0&    0&    0&    0&    0&    0&    0&    0&    0&    0&  276&    0\\
0&    0&    0& -616& -348&  344&  156& -202&  156&  414&  208& -680&   -8&  692& -572& -634\\
0&    0&    0&  244&  -36& -288& -566&  405& -566&  161&   12& 1408& -280& -252&  542&  971\\
0&    0&    0&  276&  552&   92&   46&   23&   46& -253& -368&   92&  276& -460&  138&  253\\
0&    0&    0& -380& -324&  168&  334& -219&  334&  161&  384& -668&   56&  492& -274& -553\\
0&    0&    0&  448&  504&  168&  472&  126&  196& -322& -168& -392&   56& -336&  140&  -70\\
0&    0&    0& -344& -180&  216&  712& -252&  712&   92&   60&-1148&  164&  396& -556& -964\\
0&    0&    0& -132& -252&  -84&   -6&   75&   -6&  207&   84&   12& -120&  168& -162& -195\\
0&    0&    0&  364&  444&   56& -226&  203& -226& -161& -332&  360&   80& -388&   62&  429
\end{smallarray} \right] \\[10pt]
& \gamma'(1) = {1 \over {552}}  \left[\begin{smallarray}{cccccccccccccccccccc} 
0&  552&    0&    0&    0&    0&    0&    0&    0&    0&    0&    0&    0&    0&    0&    0\\
     0&    0&    0&  552&    0&    0&    0&    0&    0&    0&    0&    0&    0&    0&    0&    0\\
     0&    0&    0&    0&    0&  552&    0&    0&    0&    0&    0&    0&    0&    0&    0&    0\\
     0&    0&    0&    0&    0&    0&    0&  552&    0&    0&    0&    0&    0&    0&    0&    0\\
     0&    0&    0&    0&    0&    0&    0&    0&    0&  552&    0&    0&    0&    0&    0&    0\\
     0&    0&    0&    0&    0&    0&    0&    0&    0&    0&    0&  552&    0&    0&    0&    0\\
     0&    0&    0&    0&    0&    0&    0&    0&    0&    0&    0&    0&    0&  552&    0&    0\\
     0&    0&    0&    0&    0&    0&    0&    0&    0&    0&    0&    0&    0&    0&    0&  552\\
     0&    0&    0& -104& -600&  -16& -724&  218& -724&  322&  568& 1632& -496&   32&  692&  942\\
     0&    0&    0& -444& -396&  420&  306&  -99&  306&  345&  408&-1164&   48&  816& -294& -957\\
     0&    0&    0& -300&  180&  612&  714& -231&  714&   69& -336&-1428&  480&  432& -870& -945\\
     0&    0&    0&  616&   72& -528&-1076&  846&-1076&  230&  -24& 2704& -544& -600& 1124& 1922\\
     0&    0&    0&  324&  468&  156&  642&  393&  642& -483&  120& -732&  144& -312&  -54& -249\\
     0&    0&    0&  376&  216& -480& -284&  606& -284&  230&  -72& 1120& -160& -696&  428&  890\\
     0&    0&    0&  772&  420& -596& -726&  565& -726& -207& -232& 1820& -352&-1016&  914& 1291\\
     0&    0&    0& -980& -516&  748& 1302&-1049& 1302&  -69&  632&-2788&  464& 1264&-1186&-2351
\end{smallarray}\right]
\quad
w' = \left[ \begin{smallarray}{c}
-3\\
 -2\\
 -1\\
 -3\\
 -6\\
 -5\\
 -1\\
 -3\\
 -6\\
 -6\\
 -3\\
 -3\\
  3\\
  1\\
 -1\\
  3
\end{smallarray} 
\right] .
\end{align*}
Now a simple queue-based algorithm \cite[\S 4.11]{Shallit:2022}
establishes that the set $S$ of all vectors of the
form $v' \gamma'(x)$ for $x \in \{0,1\}^*$ is finite and is of cardinality $268$.
From this we easily construct a DFAO $M$ computing $f(n)$:  the
initial state is $v'$, the
state-set is $S$, and for $s\in S$ the transition $\delta(s,a)$
on the symbol $a\in \{0,1\}$ is given by $s \gamma'(a)$.
The output associated with a state $s$ is $s\cdot w'$.

By inspection of $M$, we see that the range of $f$ is 
$$\{ -18, -15, -14, -12, -11, -10, -9, -8, -7, -6, -5, -4, -3, -2, -1, 0, 1, 2, 3, 4, 5, 6, 7, 8, 9, 12, 13, 18
 \}.$$
Thus we have shown $f(n) = d(n) - 4d(\lfloor n/4 \rfloor) \geq -18$.

We now verify by induction on $n$ that $d(n) > n/5 + 7 $ for $n \geq 87$.
The base case is
$87 \leq n < 348$, and is easily checked.

Now assume $n \geq 348$ and that $d(n') > n'/5 + 7$ for $87 \leq n' < n$.
Then by induction
$$d(n) \geq 4d(\lfloor n/4 \rfloor) - 18
> 4(\lfloor n/4 \rfloor/5 + 7) - 18
\geq 4 ( (n/4 - 1)/5 + 7) - 18 > n/5 + 7,$$
as desired.

After checking that $d(n) > 0$ for $0 \leq n < 87$, 
it follows that $d(n) > 0$ for all $n$.
\end{proof}

\end{document}